\documentclass{amsart}

\usepackage{graphicx}

\usepackage{amsfonts, amsmath, amscd}
\usepackage[psamsfonts]{amssymb}

\usepackage{amssymb}

\usepackage{pb-diagram}
\usepackage[usenames]{color}

\usepackage{amssymb}
\usepackage{amsfonts}
\usepackage{graphicx}
\usepackage{amscd}
\usepackage[all,cmtip]{xy}

\usepackage{amsfonts, amsmath, amscd}
\usepackage[psamsfonts]{amssymb}

\usepackage{amssymb}

\usepackage{pb-diagram}

\usepackage[usenames]{color}

\newtheorem{theorem}{Theorem}[section]
\newtheorem{lemma}[theorem]{Lemma}
\newtheorem{corollary}[theorem]{Corollary}

\newtheorem{remark}[theorem]{Remark}
\newtheorem{proposition}[theorem]{Proposition}
\newtheorem{definition}[theorem]{Definition}
\newtheorem{example}[theorem]{Example}

\newtheorem{problem}[theorem]{Problem}

\numberwithin{equation}{section}

\newcommand{\NN}{\mathbb{N}}
\newcommand{\DD}{\mathcal{D}}
\newcommand{\GG}{\mathfrak{G}}
\newcommand{\dd}{(\mathbf{D})}

\newcommand{\UU}{\mathcal{U}}
\newcommand{\FF}{\mathcal{F}}

\newcommand{\MM}{\mathbf{M}}
\newcommand{\Nn}{\mathcal{N}}

\newcommand{\nn}{\mathfrak{n}}
\newcommand{\Pp}{\mathfrak{P}}

\author{S. S. Gabriyelyan, J. K{\c{a}}kol}

\address{Department of Mathematics, Ben-Gurion University of the Negev, Beer-Sheva, P.O. 653, Israel}
\email{saak@math.bgu.ac.il}

\address{Faculty of Mathematics and Informatics, A. Mickiewicz University\\ $61-614$ Pozna{\'n}, Poland}
\email{kakol@amu.edu.pl}

\thanks{The first named author was partially supported by Israel Science Foundation grant 1/12.}
\thanks{The second named author was  supported by Generalitat Valenciana, Conselleria d'Educaci\'{o}, Cultura i Esport, Spain, Grant PROMETEO/2013/058.}

\begin{document}

\title[On topological spaces with  certain local countable networks]{On topological spaces  and topological groups  with  certain  local countable networks}

\subjclass[2000]{Primary 54E18, 54H11; Secondary 46A03,   54E52}

\keywords{$cn$-network, the strong Pytkeev property, small base, Baire space, topological group, function space, (free) locally convex space, free (abelian) topological group}

\begin{abstract}
Being motivated by the study of the space $C_c(X)$ of all continuous real-valued functions on a Tychonoff  space $X$ with the compact-open topology, we introduced in \cite{GK-GMS1} the concepts of a $cp$-network and a $cn$-network (at a point $x$) in $X$. In the present paper we describe the topology of  $X$ admitting  a countable $cp$- or $cn$-network at a point $x\in X$. This description  applies to provide new results about the strong Pytkeev property, already well recognized and applicable concept originally introduced by Tsaban and Zdomskyy \cite{boaz}. We show that a Baire topological group $G$ is metrizable if and only if $G$ has the strong Pytkeev property. We prove also that a topological group $G$ has a countable $cp$-network if and only if $G$ is separable and has a countable $cp$-network at the unit. As an application we show, among the others, that the space $D'(\Omega)$ of distributions over open $\Omega\subseteq\mathbb{R}^{n}$ has a countable $cp$-network, which  essentially improves the well known fact stating that  $D'(\Omega)$ has countable tightness. We show that, if $X$ is an $\mathcal{MK}_\omega$-space, then the free topological group $F(X)$ and the free locally convex space $L(X)$ have a countable \mbox{$cp$-network}. We prove that a topological vector space $E$ is $p$-normed (for some \mbox{$0<p\leq 1$}) if and only if $E$ is Fr\'echet-Urysohn and admits a fundamental sequence of bounded sets.
\end{abstract}

\maketitle

\section{Introduction}

All topological spaces  are assumed to be Hausdorff.
Various topological properties generalizing metrizability have been studied intensively by topologists and analysts, especially like first countability, Fr\'{e}chet--Urysohness, sequentiality and countable tightness (see \cite{Eng,kak}). It is well-known that
\[
\xymatrix{
\mbox{metric} \ar@{=>}[r] & \mbox{first}  \atop \mbox{countable} \ar@{=>}[r] & \mbox{Fr\'{e}chet--} \atop \mbox{Urysohn} \ar@{=>}[r] & \mbox{sequential} \ar@{=>}[r] & \mbox{countable}  \atop \mbox{ tight}},
\]
and none of these implications can be reversed.

One of the most immediate extensions of the class of separable metrizable spaces are the classes of cosmic and $\aleph_{0}$-spaces in sense of Michael \cite{Mich}.
\begin{definition}[see \protect\cite{Mich}]\label{def-Cosmic-Spase}
A topological space $X$ is called

$\bullet$ \emph{cosmic}, if $X$ is a regular space with a countable network (a
family $\mathcal{N}$ of subsets of $X$ is called a \emph{network} in $X$ if,
whenever $x\in U$ with $U$ open in $X$, then $x\in N\subseteq U$ for some $N
\in\mathcal{N}$);

$\bullet$ an \emph{$\aleph_0$-space}, if $X$ is a regular space with a countable $k$-network (a family $\mathcal{N}$ of subsets of $X$ is called a \emph{$k$-network} in $X$ if, whenever $K\subseteq U$ with $K$ compact and $U$ open in $X$, then $K\subseteq \bigcup \mathcal{F}\subseteq U$ for some finite family $\mathcal{F} \subseteq\mathcal{N}$).
\end{definition}
These classes of topological spaces were intensively studied in \cite{Gao}, \cite{Guth}, \cite{Mich} and \cite{Mich1}.

Having in mind the Nagata-Smirnov metrization theorem, Okuyama \cite{Oku} and O'Meara \cite{OMe2} introduced the classes of $\sigma$-spaces and $\aleph$-spaces, respectively.
\begin{definition} \label{def-Sigma-Spase}
A topological space $X$ is called

$\bullet$ (\cite{Oku}) a {\em $\sigma$-space} if $X$ is regular and has a $\sigma$-locally finite network.

$\bullet$ (\cite{OMe2}) an {\em $\aleph$-space} if $X$ is regular and has a $\sigma$-locally finite $k$-network.
\end{definition}
Any metrizable space $X$ is an \mbox{$\aleph$-space}. O'Meara \cite{OMe} proved that an \mbox{$\aleph$-space} which is either first countable or locally compact is metrizable. Every compact subset of a $\sigma$-space is metrizable \cite{Morita}. Further results see \cite{gruenhage}.

Pytkeev \cite{Pyt} proved that every sequential space satisfies the property, known actually as the {\it Pytkeev property}, which is stronger than countable tightness: a topological space $X$ has the {\it Pytkeev property} if for each \mbox{$A\subseteq X$} and each $x\in \overline{A}\setminus A$, there are infinite subsets $A_1, A_2, \dots $ of $A$ such that each neighborhood of $x$ contains some $A_n$. Tsaban and Zdomskyy \cite{boaz} strengthened this property as follows. A  topological space $X$ has  the {\it strong Pytkeev property} if for each $x\in X$, there exists a countable family $\DD$ of subsets of $Y$,   such that for each neighborhood $U$ of $x$ and  each $A\subseteq X$ with $x\in \overline{A}\setminus A$, there is $D\in\DD$ such that $D\subseteq U$ and $D\cap A$ is infinite. Next,  Banakh \cite{Banakh}  introduced the concept  of the Pytkeev network in $X$ as follows: A family $\mathcal{N}$ of subsets of a topological space $X$ is called a  {\em Pytkeev network at a point $x\in X$} if $\Nn$ is a network at $x$ and for every open set $U\subseteq X$ and a set $A$ accumulating at $x$ there is a set $N\in\Nn$ such that $N\subseteq U$ and $N\cap A$ is infinite. Hence $X$ has the strong Pytkeev property if and only if $X$ has a countable Pytkeev network at each point $x\in X$.

In \cite{GKL2} we proved that the space $C_c(X)$  has the strong Pytkeev property for every \v{C}ech-complete Lindel\"{o}f space $X$. For the proof of this result we  constructed a family $\DD$ of sets in $C_c(X)$ such that for every neighborhood $U_\mathbf{0}$ of the zero function $\mathbf{0}$ the union $\bigcup \{ D \in\DD : \mathbf{0}\in D \subseteq U_\mathbf{0}\}$ is a  neighborhood of $\mathbf{0}$ (see the condition $\dd$ in \cite{GKL2}). Having in mind this idea for $C_{c}(X)$  we proposed in \cite{GK-GMS1} the following  types of networks which will be applied in the sequel.
\begin{definition}[\cite{GK-GMS1}] \label{def2}
A family $\Nn$ of subsets of a topological space $X$ is called

$\bullet$ a {\em $cn$-network at a point $x\in X$} if for each neighborhood $O_x$ of $x$ the set $\bigcup \{ N \in\Nn : x\in N \subseteq O_x \}$ is a neighborhood of $x$; $\Nn$ is a {\em $cn$-network} in $X$ if $\mathcal{N}$ is a $cn$-network at each point $x\in X$.

$\bullet$  a {\em $ck$-network at  a point} $x\in X$ if for any neighborhood $O_x$ of $x$ there is a neighborhood $U_x$ of $x$ such that for each compact subset $K\subseteq U_x$ there exists a finite subfamily $\FF\subseteq\mathcal{N}$ satisfying $x\in \bigcap\FF$ and $K\subseteq\bigcup\FF\subseteq O_x$; $\Nn$ is a {\em $ck$-network} in $X$ if $\mathcal{N}$ is a $ck$-network at each point $x\in X$.

$\bullet$  a {\em $cp$-network at  a point}  $x\in X$ if $\Nn$ is a network at $x$ and for any subset $A\subseteq X$ with $x\in \overline{A}\setminus A$ and each neighborhood $O_x$ of $x$ there is a set $N\in\mathcal{N}$ such that $x\in N\subseteq O_x$ and $N\cap A$ is infinite; $\Nn$ is  a  {\em $cp$-network} in $X$ if $\mathcal{N}$ is a Pytkeev network at each point $x\in X$.
\end{definition}
Therefore  the notion   of the  $cn$-network naturally connects the notions of the network and the base at a point $x\in X$.

Now Definitions \ref{def-Cosmic-Spase}, \ref{def-Sigma-Spase} and \ref{def2} motivated the following concepts.
\begin{definition} \label{def-Pytkeev-Net}
A topological space $X$ is called

$\bullet$ (\cite{Banakh}) a {\em $\Pp_0$-space} if $X$ has a countable Pytkeev network.

$\bullet$ (\cite{GK-GMS1})  a {\em $\Pp$-space} if $X$ has a $\sigma$-locally finite $cp$-network.
\end{definition}
It is known that: $\Pp_0$-space $\Rightarrow$ $\aleph_0$-space $\Rightarrow$ cosmic; but the converse is false (see \cite{Banakh,Mich}). Each $\Pp$-space $X$ has the strong Pytkeev property and is an $\aleph$-space \cite{GK-GMS1}.

Definition \ref{def2} allow us to define the following cardinals of topological spaces.

\begin{definition}[\cite{GK-GMS1}]
Let $x$ be  a point of a topological space $X$ and $\nn\in\{ cp, ck,  cn\}$. The smallest size $|\mathcal{N}|$ of an $\nn$-network $\mathcal{N}$ at $x$ is called the {\em $\nn$-character of $X$ at the point $x$} and is denoted by $\nn_\chi(X,x)$. The cardinal $\nn_\chi(X)=\sup\{ \nn_\chi(X,x): x\in X\}$ is called the {\em $\nn$-character} of  $X$. The {\em \mbox{$\nn$-netweight}}, $\nn w(X)$, of $X$ is the least cardinality of an $\nn$-network for $X$.
\end{definition}

Our first result is a characterization of cosmic, $\aleph_0$- and $\Pp_0$-groups $G$ by their separability and countability of the corresponding types of networks at the unit $e\in G$. Recall (\cite{Mich}) that every cosmic space is (even hereditary) separable. So it is natural to ask: For $\nn\in\{ cp, ck,  cn\}$, does any separable space $X$ of countable $\nn$-character have also a countable $\mathfrak{n}$-network? In general the answer is negative. Indeed, any first countable separable non-metrizable compact space $K$ trivially has countable $cp$-character, but $K$ is not even a $\sigma$-space because it is not metrizable.  However, for the group case we have the following
\begin{theorem} \label{t-Separable}
Let $G$ be a topological group. Then $G$ is a cosmic space (resp. an $\aleph_0$-space or a $\Pp_0$-space) if and only if $G$ is separable and has countable $cn$-character (resp. countable $ck$-character or  countable $cp$-character).
\end{theorem}

In \cite{GKL2} we proved that the space of distributions $\mathfrak{D}'(\Omega)$ over an open set $\Omega\subset\mathbb{R}^{n}$ has the strong Pytkeev property. Since $\mathfrak{D}'(\Omega)$ is also separable, Theorem \ref{t-Separable} strengthens this result as follows.
\begin{corollary} \label{c:Distr-Po}
The space of distributions $\mathfrak{D}'(\Omega)$ is a $\Pp_0$-space.
\end{corollary}
For other applications of Theorem \ref{t-Separable},  see Corollaries \ref{c:Cosmic-LCS} and \ref{c:Frechet-P0}.

One of the most interesting and essential problems while studying a certain class of topological spaces  is to describe the topology of spaces from this class. Recently we have described in \cite{GKKLP} the topology $\tau$ of any cosmic space (or $\aleph_0$-space and  $\Pp_0$-space) $(X,\tau)$ in term of a ``small base'' $\UU(\tau)$.
Since the class of spaces with a countable $cn$-character (or a countable $ck$- and $cp$-character) provides the most natural {\it local} generalization of cosmic spaces (or $\aleph_0$-spaces or $\Pp_0$-spaces, respectively), it is natural to ask whether such a description of the topology (as mentioned above) can be also obtained for this  wider class. This question is of an independent interest for the class of topological groups $G$ because the topology of $G$ is defined  essentially by the filter of open neighborhoods at the unit $e$ of $G$.
The main result of the paper is to give a positive answer to this question. For this purpose we need the following concepts.

Let $\Omega$ be a set and $I$ be a partially ordered set with an order $\leq$. We say that a family $\{ A_i\}_{i\in I}$ of subsets of $\Omega$ is  {\it $I$-decreasing}  if $A_j \subseteq A_i$  for every $i\leq j$ in $I$.
One of the most important example of partially ordered sets is the product  $\NN^\NN$  endowed with the natural partial order, i.e., $\alpha\leq\beta$ if $\alpha_i \leq\beta_i$ for all $i\in\NN$, where $\alpha=(\alpha_i)_{i\in\NN}$ and $\beta=(\beta_i)_{i\in\NN}$.
For every $\alpha=(\alpha_i)_{i\in\NN} \in\NN^\NN$ and each $ k\in\NN$, set
$I_k(\alpha) := \left\{ \beta\in\NN^\NN : \beta_i = \alpha_i \mbox{ for } i=1,\dots,k\right\}$.
Let $\mathbf{M}\subseteq \NN^\NN$ and $\mathcal{U}=\{ U_\alpha : \alpha\in \mathbf{M}\}$ be an $\mathbf{M}$-decreasing family of subsets of a set $\Omega$. Then we define the countable family $\mathcal{D}_\mathcal{U}$ of subsets of $\Omega$ by
\[
\mathcal{D}_\mathcal{U} :=\{ D_k (\alpha): \; \alpha\in \mathbf{M} ,\ k\in\NN \}, \mbox{ where } D_{k}(\alpha) = \bigcap_{\beta\in I_{k}(\alpha)\cap\mathbf{M}} U_\beta,
\]
and say that $\mathcal{U}$ satisfies the {\it condition $\dd$} if $U_\alpha =\bigcup_{k\in\NN} D_k (\alpha)$ for every $\alpha\in \mathbf{M}$. A similar condition naturally appears and is essentially used in \cite{GaK,GKKLP,GKL,GKL2}.

\begin{definition}[\cite{GKKLP}] \label{def11}
A topological space  $(X,\tau)$ has a {\em small base } if there exists an $\mathbf{M}$-decreasing base of $\tau$ for some $\mathbf{M}\subseteq \NN^\NN$.
\end{definition}

The above notion has been used to describe the topology of cosmic spaces, $\aleph_{0}$-spaces and $\Pp_{0}$-spaces, respectively. The  item (iii) immediately follows from the proof of (i) and (ii) of  Theorem \ref{quasiC} given in \cite{GKKLP}.
\begin{theorem}[\cite{GKKLP}] \label{quasiC}
Let $(X,\tau)$ be a  topological space. Then:
\begin{itemize}
\item[{\rm (i)}] $X$ has a countable network (and is cosmic)  if and only if $X$ has a  small base $\mathcal{U}=\{ U_\alpha : \alpha\in \mathbf{M}\}$ satisfying the condition $\dd$ (and is regular). In that case the family $\mathcal{D}_\mathcal{U}$ is a countable network in $X$.
\item[{\rm (ii)}] $X$ has a countable $k$-network (and is an $\aleph_0$-space)   if and only if $X$ has a  small base $\mathcal{U}=\{ U_\alpha : \alpha\in \mathbf{M}\}$ satisfying the condition $\dd$ such that  the family $\mathcal{D}_\mathcal{U}$ is a countable $k$-network in $X$ (and is regular).
\item[{\rm (iii)}] $X$ has a countable $cp$-network (and is a $\Pp_0$-space)   if and only if $X$ has a  small base $\mathcal{U}=\{ U_\alpha : \alpha\in \mathbf{M}\}$ satisfying the condition $\dd$ such that  the family $\mathcal{D}_\mathcal{U}$ is a countable $cp$-network in $X$ (and is regular).
\end{itemize}
For this three cases we can find a small base $\mathcal{U}$ such that $U_{\alpha}\not= U_\beta$ for $\alpha\not= \beta$ and $\mathcal{U}=\tau$, what means that for every $W\in\tau$ there exists $\alpha\in\mathbf{M}$ such that $W=U_{\alpha}$.
\end{theorem}
Note that the condition $\dd$ is essential  in Theorem \ref{quasiC}: the Bohr compactification $b\mathbb{Z}$ of the discrete group $\mathbb{Z}$ of integers has a small base  \cite{GKKLP}, but the compact group $b\mathbb{Z}$ is not cosmic since it is not metrizable.

The following local version of  the concept from Definition \ref{def11}  will play an essential role in this paper.
\begin{definition} \label{def13}
Let $x$ be a point in  a topological space $(X,\tau)$. We say that $X$ has a {\em small base at $x$} if there exists an $\MM_x$-decreasing base at $x$ for some  $\mathbf{M}_x \subseteq \NN^\NN$.
\end{definition}
If a topological space $X$ has an $\MM_x$-decreasing base at $x$, we shall also say that the space $X$ has a (local) {\it $\MM_x$-base} at $x$. Clearly, if $G$ is a topological group with an $\MM_e$-base at the unit $e$, then $g\MM_e$ is a small base at each point $g\in G$. So we shall say simply that the group $G$ has a (local) {\it $\MM$-base} omitting the subscript $e$.  A number of  specialists in their papers dealing with  locally convex spaces ({\it lcs} for short) used the notation of {\it $\mathfrak{G}$-base} for a local $\NN^\NN$-base at zero (see, for example, \cite{CO,kak}). In this special case (i.e., when $\MM_x =\NN^\NN$) we shall follow the same notation. Topological groups with a $\GG$-base are thoroughly studied in  \cite{GKL} (see also \cite{GaK,GKL2}). Note that any metrizable group $G$ has a $\GG$-base.

Below we describe the topology of a topological space $X$ at a point $x$ in which it has countable $cn$-, $ck$- or $cp$-character.
\begin{theorem}\label{quasiP}
Let $x$ be a point of a  topological space $X$. Then:
\begin{itemize}
\item[{\rm (i)}] $X$ has a countable $cn$-network at $x$  if and only if $X$ has a  small base $\mathcal{U}(x)=\{ U_\alpha : \alpha\in \mathbf{M}_x\}$ at $x$ satisfying the condition $\dd$. In that case the family $\mathcal{D}_{\mathcal{U}(x)}$ is a countable $cn$-network  at $x$.
\item[{\rm (ii)}] $X$ has a countable $ck$-network at $x$   if and only if $X$ has a  small base $\mathcal{U}(x)=\{ U_\alpha : \alpha\in \mathbf{M}_x \}$  at $x$ satisfying the condition $\dd$ such that  the family $\mathcal{D}_{\mathcal{U}(x)}$ is a countable $ck$-network  at $x$.
\item[{\rm (iii)}] $X$ has a countable $cp$-network  at $x$  if and only if $X$ has a  small base $\mathcal{U}(x)=\{ U_\alpha : \alpha\in \mathbf{M}_x\}$ at $x$ satisfying the condition $\dd$ such that  the family $\mathcal{D}_{\mathcal{U}(x)}$ is a countable $cp$-network  at $x$.
\end{itemize}
\end{theorem}

The main application of the above theorem is the next theorem which  provides a  characterization of metrizability for topological groups with the Baire property; this  also gives a positive answer to Question 10 of \cite{GKL2} and a partial positive answers to Question 4.2 of \cite{GKL} and Question 9 of \cite{GKL2}.
\begin{theorem} \label{tMetr-cn}
Let $G$ be a Baire topological group. Then the following are equivalent:
\begin{itemize}
\item[{\rm (i)}] $G$ is metrizable.
\item[{\rm (ii)}] $G$ has the strong Pytkeev property.
\item[{\rm (iii)}]  $G$ has countable $ck$-character.
\item[{\rm (iv)}]  $G$ has countable $cn$-character.
\item[{\rm (v)}] $G$ has a $\GG$-base satisfying the condition $\dd$.
\end{itemize}
\end{theorem}
Since a regular topological space is cosmic if and only if it has a countable $cn$-network,  we have  the following
\begin{corollary}[\cite{GKKLP}] \label{cosmic}
A Baire separable topological group $G$ is metrizable if and only if $G$ is cosmic.
\end{corollary}
It turns out that even Fr\'echet-Urysohn lcs which are Baire need not to have a countable $cn$-character, see Remark \ref{rema} below.

Section \ref{secTVS} contains some applications of presented results  for the class of topological vector spaces and  for the  free (abelian) topological groups $F(X)$ ($A(X)$, respectively), as well as  for the free locally convex space $L(X)$ over a Tychonoff space $X$.
This section deals  with a natural question:  For which topological spaces $X$ the free groups $A(X)$, $F(X)$ and $L(X)$ have a \mbox{$\mathfrak{G}$-base}, countable $\nn$-character or are $\aleph_0$-spaces?  As usual, $\chi(X)$ denotes the character of a topological space $X$. Recall also that a topological space $X$ is called an {\it $\mathcal{MK}_\omega$-space} if the topology of $X$ is defined by an increasing sequence of compact metrizable subsets. Denote by $\mathfrak{d}$ the cofinality of the partially ordered set $\NN^\NN$. The next theorem gives a partial answer to the aforementioned question and provides an alternative and simple proof of the  equality $\chi(A(X))=\chi(F(X))=\mathfrak{d}$ for a non-discrete $\mathcal{MK}_\omega$-space $X$ (which  is one of the principal results of \cite{NTk}). Also this theorem generalizes Theorem 4.16 of  \cite{GKL} and gives an affirmative answer to Question 4.17 of  \cite{GKL}.
\begin{theorem} \label{tFree}
Let $X$ be an $\mathcal{MK}_\omega$-space. Then:
\begin{enumerate}
\item[{\rm (i)}] $A(X)$, $F(X)$ and $L(X)$ have a $\mathfrak{G}$-base satisfying the condition $\dd$ and are $\Pp_0$-spaces.
\item[{\rm (ii)}]
  \begin{enumerate}
  \item[{\rm (a)}]  If $X$ is not discrete, then $\chi(A(X))=\chi(F(X))=\chi(L(X))= \mathfrak{d}$.
  \item[{\rm (b)}]  If $X$ is discrete, then $\chi(A(X))=\chi(F(X))=1$, and $\chi(L(X))=\aleph_0$ if $X$ is finite and $\chi(L(X))=\mathfrak{d}$ if $X$ is infinite.
  \end{enumerate}
\end{enumerate}
\end{theorem}

At the end of Section \ref{secTVS} we mostly deal with topological vector spaces ({\em tvs} for short) having a fundamental sequence of bounded sets. In particular,  we  extend Theorem 5.1 in \cite{kaksax}  from the class of locally convex spaces to the class of tvs.
\begin{theorem} \label{t-Metr-TVS-FU}
A tvs $E$ is  $p$-normed for some $0<p\leq 1$ if and only if $E$ is Fr\'echet-Urysohn and admits a  fundamental sequence of bounded sets.
\end{theorem}

\section{Proofs of Theorems \ref{t-Separable}, \ref{quasiP} and \ref{tMetr-cn} } \label{sec-Top}

We start this section with the proof of Theorem \ref{t-Separable}.
\begin{proof}[Proof of Theorem \ref{t-Separable}]
We assume that $G$ is not discrete since, otherwise, the theorem is trivial.

Necessity is clear. Let us prove sufficiency. Let $\DD =\{ D_n \}_{n\in\NN}$ be a countable $\nn$-network at the unit $e$ of $G$ and let $\{ g_n\}_{n\in\NN}$ be a dense subset of $G$. Without loss of generality we can assume that $\DD$ is closed under taking finite products. We show that the countable family
\[
\Nn := \{ g_n D_m :\ n, m \in\NN \}
\]
is an $\nn$-network in $G$.

Fix $g\in G$ and let $gU$ be an open neighborhood of $g$. Take an open symmetric neighborhood $W$ of $e$ such that $W^3 \subseteq U$. In all three cases $\DD$ is also a $cn$-network at $e$. Hence the set $$W_0 := \bigcup \{ D\in \DD : e\in D\subseteq W\}$$ is a neighborhood of $e$. As $G=\cup_n g_n W_0$, we can find $r,t\in\NN$ such that $g=g_r \cdot h$ and $h\in D_t \subseteq W_0$.

(1) Assume that $\DD$ is a $cn$-network at $e$. Clearly,
\[
\bigcup \{ g_r D_t\cdot D_m : \ D_m \in\DD, D_m \subseteq W \} =g_r D_t\cdot W_0 \subseteq g_r W_0^2 \subseteq g W^3 \subseteq gU,
\]
and $g=g_r h \in \bigcap \{ g_r D_t\cdot D_m : \ D_m \in\DD, D_m \subseteq W \}$. So $\Nn$ is a $cn$-network at $g$.

(2) Assume that $\DD$ is a $ck$-network at $e$. Take an open neighborhood $W_1 \subseteq W$ of $e$ such that for every compact subset $K$ of $W_1$ there exists a finite subfamily $\FF$ of $\DD$ satisfying $e\in \bigcap\FF$ and $K\subseteq \bigcup \FF \subseteq W.$ As $G=\cup_n g_n W_1$, we can take $a,b\in\NN$ such that $g=g_a \cdot h$ and $h\in D_b \subseteq W$. Now, for each compact subset $gK$ of $gW_1$ we have
\[
g_aD_b\cdot\FF \subseteq\Nn,\,\,\, g\in \bigcap g_aD_b\cdot\FF
\]
 and
\[
gK=g_a h\cdot K \subseteq \bigcup g_aD_b\cdot\FF.
\]
 Thus $\Nn$ is a $ck$-network at $g$.

(3) Assume that $\DD$ is a $cp$-network at $e$. Let  $A\subseteq G$ be such that $g\in \overline{A}\setminus A$.
Since $e\in \overline{g^{-1}A}\setminus g^{-1}A$, there is $D_s\in \DD$ such that $e\in D_s \subseteq W_0$ and $D_s \cap g^{-1}A$ is infinite. So $$g_r h D_s \cap A \subseteq g_r (D_t\cdot D_s) \cap A$$ is infinite.  As $g\in g_r (D_t\cdot D_s) \in \Nn$ and
\[
g_r (D_t\cdot D_s) = g( h^{-1} \cdot D_t\cdot D_s) \subseteq g \cdot W_0^{-1} \cdot W_0^2 \subseteq g \cdot W^3 \subseteq g \cdot U,
\]
we obtain that $\Nn$ is a $cp$-network at $g$.
\end{proof}

As a corollary of this theorem we obtain the following extensions of some results from \cite{GKL2} in the class of {\em separable} locally convex spaces.
\begin{corollary} \label{c:Cosmic-LCS}
Let $E$ be a separable lcs satisfying one of the following conditions:
\begin{enumerate}
\item[{\rm (i)}] $E$ is a $(DF)$-space with countable tightness;
\item[{\rm (ii)}] $E$ is a sequential dual metris space;
\item[{\rm (iii)}] $E$ is a strict $(LM)$-space;
\item[{\rm (iv)}] $E$ is a quasibarrelled lcs with a $\GG$-base.
\end{enumerate}
Then $E$ is a $\Pp_0$-space.
\end{corollary}

\begin{proof}
By Theorems 3 and 9 of \cite{GKL2}, the space $E$ has the strong Pytkeev property, so Theorem \ref{t-Separable} applies.
\end{proof}

In Corollary 9 of \cite{GKL2} we proved that the strong dual $F'$ of a Fr\'{e}chet space $F$ has countable tightness if and only if $F'$ has the strong Pytkeev property. If additionally $F'$ is separable, Theorem \ref{t-Separable}  implies
\begin{corollary} \label{c:Frechet-P0}
Let $F$ be a Fr\'{e}chet space whose strong dual $F'$ is separable. Then the strong dual $F'$ has countable tightness if and only if $F'$ is an \mbox{$\Pp_0$-space}.
\end{corollary}

In what follows we need the following result due to Banakh \cite{Banakh} (its independent proof is given in \cite[Corollary 6.4]{GK-GMS1}).
\begin{theorem}[\cite{Banakh}] \label{c-Banakh}
If $X$ is an $\aleph_0$-space and $Y$ is a  $\Pp_0$-space, then $C_c(X,Y)$ is a $\Pp_0$-space.
\end{theorem}

Theorems \ref{t-Separable} and \ref{quasiC} with the comments after it, may suggest also the following
\begin{example}\label{jointly}
There exists a large class of lcs $E$ for which the existence of a  $\mathfrak{G}$-base (the condition $\dd$ is not required) in $E$ implies that  $E$ is a $\Pp_{0}$-space.
\end{example}
\begin{proof}
We prove the following two facts. (i) If $C_c(X)$ is a separable space admitting a $\mathfrak{G}$-base, then $C_{c}(X)$ is a $\Pp_{0}$-space. (ii) There exist however a $\Pp_{0}$-spaces  $C_{c}(X)$ (hence having  countable $cp$-character) which  do not admit a $\mathfrak{G}$-base.
We need the following main result of \cite{feka}: $C_{c}(X)$ has a $\mathfrak{G}$-base if and only if $X$ admits a compact resolution swallowing compact sets, i.e. a family $\{K_{\alpha}:\alpha\in\mathbb{N}^{\mathbb{N}}\}$ of compact sets covering $X$ such that $K_{\alpha}\subseteq K_{\beta}$ for all $\alpha\leq\beta$ in $\mathbb{N}^{\mathbb{N}}$ and each compact set of $X$ is contained in some $K_{\alpha}$.
Since $C_{c}(X)$ is separable, it admits a weaker metrizable and separable topology; hence $X$ admits such a weaker topology. By the above remark we conclude that $X$ has a compact resolution swallowing compact sets. Then applying \cite[Theorem 3.6]{COT} the space $X$ is an $\aleph_{0}$-space. Hence $C_{c}(X)$ is a $\Pp_{0}$-space by Theorem \ref{c-Banakh}. To complete the claim (ii) let $X=\mathbb{Q}$ be the space of rational numbers. Theorem \ref{c-Banakh} implies that $C_c(\mathbb{Q})$  is  a $\Pp_{0}$-space but (as $\mathbb{Q}$ does not have a compact resolution swallowing compact sets) it does not have a $\mathfrak{G}$-base.
\end{proof}

Below we prove Theorem \ref{quasiP}.
\begin{proof} [Proof of Theorem \ref{quasiP}]

The idea of the proof uses some arguments from the proof of Theorem 1.3 of \cite{GaK}.
If $x$ is an isolated point, we set $\mathbf{M}_x :=\NN^\NN$ and $U_\alpha :=\{ x\}$ for each $\alpha\in\mathbf{M}_x$. Clearly, the family $\{ U_\alpha: \alpha\in \mathbf{M}_x\}$ is as desired in all three cases of the theorem. So we shall assume that $x$ is not isolated.

(i) Assume that $X$ has a countable $cn$-network $\mathcal{D}=\{ D_i\}_{i\in\NN}$ at $x$.
Recall that $D_i$ contains $x$ for every $i\in\NN$.

{\it Step} 1.
For every $k,i\in\NN$, set
\[
D_k^i := \bigcap_{l=1}^{k} D_{i-1+l}.
\]
So, for each $i\in\NN$, the sequence $\{ D_k^i \}_{k\in\NN}$ is decreasing. For every $\alpha=(\alpha_i)_{i\in\NN}\in\NN^\NN$, set
\[
A_\alpha := \bigcup_{i\in\NN} D^i_{\alpha_i} = \bigcup_{i\in\NN} \bigcap_{l=1}^{\alpha_i} D_{i-1+l} .
\]
Clearly, $x\in A_\alpha$ and $A_\alpha \subseteq A_\beta$ for each $\alpha,\beta\in\NN^\NN$ with $\beta\leq\alpha$.

{\it Step} 2. Let $V$ be a neighborhood of $x$. Set $J(V):=\{ j\in\NN : D_j \subseteq V\}$. Since $x$ is not isolated, the family $J(V)$ is infinite. Now we prove that  the following  condition holds.
\begin{itemize}
\item[$\mathbf{(A)}$] If $W$ is a neighborhood of $x$ and $J(W):=\{ j\in\NN : D_j \subseteq W\} =\{ n_k\}_{k\in\NN}$ with $n_1 <n_2< \dots$, then there is $\alpha=\alpha(W) \in \NN^\NN$ such that
  \begin{itemize}
  \item[{\rm ($A_1$)}] $\alpha_{n_k} = 1$ for every $k\in\NN$;
  \item[{\rm ($A_2$)}] $A_\alpha = \bigcup_{k\in\NN} D_{n_k} (\subseteq W)$ is a neighborhood of $x$.
  \end{itemize}
\end{itemize}
We construct  $\alpha=\alpha(W)$ as follows. If $i=n_k$ for some $k\in\NN$ we set $\alpha_i =1$. So $D^i_{\alpha_i}= D_{n_k}$. Set $n_0 :=0$. Now, if $n_{k-1} <i<n_k$ for some $k\in\NN$, we set $\alpha_i := n_k - i+1$. Then
\[
D^i_{\alpha_i} = \bigcap_{l=1}^{\alpha_i} D_{i-1+l}\subseteq D_{i-1+\alpha_i} = D_{n_k}.
\]
Hence $A_\alpha = \bigcup_{k\in\NN} D_{n_k}$. Since $\mathcal{D}$ is a $cn$-network at $x$, $A_\alpha$ is  a neighborhood of $x$. Thus ($A_1$) and ($A_2$) are satisfied.

{\it Step} 3. Denote by $\mathbf{M}_x$ the set of all $\alpha \in  \NN^\NN$  of the form $\alpha=\alpha(W)$ for some neighborhood $W$ of $x$. For each $\alpha\in\mathbf{M}_x$ set $U_\alpha :=A_\alpha$. Now, by $\mathbf{(A)}$, the family $\{ U_\alpha : \alpha\in\mathbf{M}_x\}$ is a small base at $x$.

{\it Step} 4. Now we check that the condition $\dd$ holds.
It is clear that $\bigcup_{k\in\NN} D_k (\alpha)\subseteq U_\alpha$. We prove the converse inclusion as follows
\[
\begin{split}
\bigcup_{k\in\NN} D_k (\alpha) & = \bigcup_{k\in\NN} \bigcap_{\beta\in I_k(\alpha)\cap\mathbf{M}_x} U_\beta \\
& = \bigcup_{k\in\NN} \bigcap_{\beta\in I_k(\alpha)\cap\mathbf{M}_x} \left( \bigcup_{i\in\NN} \bigcap_{l=1}^{\beta_i} D_{i-1+l}\right) \; (\mbox{take only } i=k) \\
  & \supseteq \bigcup_{k\in\NN} \bigcap_{\beta\in I_k(\alpha)\cap\mathbf{M}_x}  \left(  \bigcap_{l=1}^{\beta_k} D_{k-1+l} \right) \; (\mbox{since } \beta_k =\alpha_k) \\
  & = \bigcup_{k\in\NN} \bigcap_{l=1}^{\alpha_k} D_{k-1+l} = U_\alpha .
\end{split}
\]

Conversely, if $X$ has a small base at $x$ satisfying the condition $\dd$, then clearly the countable family $\mathcal{D}_{\mathcal{U}(x)}$ is a $cn$-network at $x$.

(ii) Assume that $X$ has a countable $ck$-network  $\mathcal{D}=\{ D_i\}_{i\in\NN}$ at $x$. Without loss of generality we  may assume that $\mathcal{D}$  is closed under taking finite unions. Similarly as in item (i)  we can prove that $X$ has a small base $\{ U_\alpha : \alpha\in \mathbf{M}_x\}$ at $x$ satisfying the condition $\dd$. We show that the countable family $\mathcal{D}_{\mathcal{U}(x)}$ is also a $ck$-network at $x$.

Let $O_x$ be a neighborhood of $x$. Set $W:= \bigcup_{j\in J(O_x)} D_j$.
Since $\mathcal{D}$ is a $ck$-network at $x$, there is a neighborhood $U_x \subseteq O_x$ of $x$ such that for each compact subset $K\subseteq U_x$ there is $j\in J(O_x)$ such that $K\subseteq D_j \subseteq W\subseteq O_x$.  So $U_x \subseteq W$ and hence $W$ is a neighborhood of $x$.

Let $K$ be a compact subset of $U_x$. Set $\alpha = \alpha(O_x)$. By the construction of $W$, there exists $i\in J(O_x)$ such that $K\subseteq D_i \subseteq W$.  By the definition of $J(O_x)$, we have $i=n_k$ for some $k\in\NN$. So $x\in D_i =D_{n_k}$ and $\alpha_{n_k}=1$ by  $(A_1)$. As
\[
\begin{split}
D_{n_k}(\alpha) & = \bigcap_{\beta\in I_{n_k}(\alpha)\cap\mathbf{M}_x} A_\beta = \bigcap_{\beta\in I_{n_k}(\alpha)\cap\mathbf{M}_x} \left( \bigcup_{i\in\NN} \bigcap_{l=1}^{\beta_i} D_{i-1+l}\right) (\mbox{take } i=n_k) \\
& \supseteq \bigcap_{\beta\in I_{n_k}(\alpha)\cap\mathbf{M}_x}  \left(  \bigcap_{l=1}^{\beta_{n_k}} D_{n_k -1+l} \right) \; (\mbox{since } \beta_{n_k} =\alpha_{n_k}=1 ) = D_{n_k},
\end{split}
\]
we obtain that $K\subseteq D_i \subseteq D_{n_k}(\alpha) \subseteq W$. Thus $\mathcal{D}_{\mathcal{U}(x)}$ is a countable \mbox{$ck$-network} at $x$.

The converse assertion is clear.

(iii) Assume that $X$ has a countable $cp$-network  $\mathcal{D}=\{ D_i\}_{i\in\NN}$ at $x$. Without loss of generality we may also assume that  $\mathcal{D}$  is closed under taking finite unions. As in item (i), we  prove that $X$ has a small base $\{ U_\alpha : \alpha\in \mathbf{M}_x\}$ at $x$ satisfying the condition $\dd$. We show that the countable family $\mathcal{D}_{\mathcal{U}(x)}$ is also a $cp$-network at $x$.

Let $A\subseteq X$ with $x\in \overline{A}\setminus A$ and let $O_x$ be a neighborhood  of $x$. Set $\alpha = \alpha(O_x)$ and $W:= \bigcup_{j\in J(O_x)} D_j$. Since $\DD$ is also a $cn$-network, $W$ is a neighborhood of $x$. By the construction of $W$ and the definition of \mbox{$cp$-network}, there exists $i\in J(O_x)$ such that $x\in D_i \subseteq W$ and $D_i \cap A$ is infinite.  By the definition of $J(O_x)$, we have $i=n_k \in J(O_x)$ for some $k\in\NN$. So $D_i =D_{n_k}$. In (ii) we proved that $D_i \subseteq D_{n_k}(\alpha)$. So $A \cap D_{n_k}(\alpha)$ is infinite. Thus $\mathcal{D}_{\mathcal{U}(x)}$ is a countable $cp$-network at $x$.

The converse assertion is trivial.
\end{proof}

\begin{remark} {\em
In Proposition 2 of \cite{GKL2} it is shown that there is a topological group $G$ with a $\GG$-base which has uncountable tightness (see also Example \ref{exa-Cc(b)} below). So $cn_\chi (G) >\aleph_0$ (see \cite{GK-GMS1}).
Also the compact group $b\mathbb{Z}$ has a small base  by \cite{GKKLP}, and $\chi(b\mathbb{Z})=2^{\aleph_0}=\mathfrak{c}$. }
\end{remark}

It is somewhat  surprising   that the validity of the condition $\dd$ essentially depends on the chosen family $\mathbf{M}_x$ as the following example shows.

\begin{example} \label{exa1} {\em
We consider the Banach  separable space $\ell^1$ and build a small base at $\mathbf{0}$ in $\ell^1$ as follows. Set $\MM_0 := \NN^\NN \cap \ell^\infty$. For every $\alpha=(\alpha_i)_{i\in\NN} \in \mathbf{M}_0$, set
\[
U_\alpha =\left\{ (x_i)_{i\in\NN} \in \ell^1 : \; \sum_{i} \alpha_i |x_i| <1 \right\}.
\]
Clearly, $\{ U_\alpha : \alpha\in\mathbf{M}_0 \}$ is a  small base in $\ell^1$.
For each $\alpha=(\alpha_i)\in \mathbf{M}_0$  and every $k\in\NN$ we have
\[
D_k(\alpha) =\{ (x_i)\in\ell^1 : \alpha_1|x_1| +\cdots + \alpha_k|x_k| <1 \mbox{ and } 0= x_{k+1} =x_{k+2}=\dots \},
\]
and $\bigcup_{k\in\NN} D_k(\alpha)\not= U_\alpha$. So condition $\dd$ does not hold. }
\end{example}

We shall need the following
\begin{proposition}[\cite{Banakh1}] \label{fB}
Any countable $cp$-network at a point $x$ of a topological space $X$ is  a $ck$-network at $x$.
\end{proposition}

Now we are ready to prove Theorem \ref{tMetr-cn}.
\begin{proof} [Proof of Theorem \ref{tMetr-cn}]

The implications (i)$\Rightarrow$(ii) and (iii)$\Rightarrow$(iv)  are clear, (ii)$\Rightarrow$(iii) follows from Proposition \ref{fB}, and (v)$\Rightarrow$(iv) follows from Theorem \ref{quasiP}(i).

(i)$\Rightarrow$(v) If $\{ V_n\}_{n\in\NN}$  is a decreasing base of neighborhoods at the unit $e$ of $G$, then the family $\{ U_\alpha : \alpha\in \NN^\NN\}$, where $U_{\alpha}:=V_{\alpha_{1}}$  for $\alpha=(\alpha_{i})\in\mathbb{N}^{\mathbb{N}}$, is a $\GG$-base satisfying the condition $\dd$.

(iv)$\Rightarrow$(i)
Let $G$ have a countable $cn$-character. We have to show that $G$ is metrizable. We prove that $G$ has  a countable base of neighborhoods at the unit $e$. By Theorem \ref{quasiP}(i) there exists a small local base $\mathcal{U}=\{U_\alpha : \alpha\in \mathbf{M}\}$ at $e$ satisfying the condition $\dd$. We show that the countable family $\{\overline{D_{k}(\alpha)}\cdot\overline{D_{k}(\alpha)}^{\ -1}:\alpha\in\mathbf{M}, k\in\NN\}$ contains a base of neighborhoods of  $e$ in $G$. Indeed,  let $W$ be an open neighborhood of $e$. Choose a symmetric open neighborhood $V$ of $e$ such that $V\cdot V \subseteq \overline{V}\cdot\overline{V}\subseteq W$. There exists $\alpha\in\mathbf{M}$ with $U_{\alpha}=\bigcup_{k}D_{k}(\alpha)\subseteq V$.  Since $\mathrm{Int}(U_{\alpha})$ is open in $G$ and $G$ is Baire, there exists $k\in\NN$ such that $\mathrm{Int}(U_{\alpha})\cap\overline{D_{k}(\alpha)}$ has a non-empty interior in $U_{\alpha}$, so also in $G$. Therefore $\overline{D_{k}(\alpha)}\cdot\overline{D_{k}(\alpha)}^{\ -1}$ is a neighborhood of $e$ which is  contained in $W$.
\end{proof}
We do not know whether the assumption on a $\GG$-base to satisfy the condition $\dd$ can be omitted in Theorem \ref{tMetr-cn}(v). However, Example \ref{exa1} shows that the condition $\dd$ essentially depends on the chosen family $\mathbf{M}_x$. This  suggests the next question:
\begin{problem} \label{qG-D}
Let $G$ be a Baire topological group with a $\GG$-base $\mathcal{U}$. Does $\mathcal{U}$ necessarily satisfy the condition $\dd$?
\end{problem}


\section{Applications to free (abelian) topological groups and topological vector spaces} \label{secTVS}

Now we apply the obtained results to the important classes of free lcs and free (abelian) topological groups.
The following  concept is due to Markov \cite{Mar}, see also Graev \cite{Gra}.
\begin{definition}
Let $X$ be a Tychonoff space. A topological group $F(X)$ (respectively, $A(X)$) is called {\em  the (Markov) free ({\em respectively}, abelian) topological  group} over  $X$ if $F(X)$ (respectively, $A(X)$) satisfies the following conditions:
\begin{enumerate}
\item[{\rm (i)}] There is a continuous mapping $i: X\to F(X)$ (respectively, $i: X\to A(X)$) such that $i(X)$ algebraically generates $F(X)$ (respectively, $A(X)$).
\item[{\rm (ii)}] If $f: X\to G$ is a continuous mapping to a (respectively, abelian) topological  group $G$, then there exists a continuous homomorphism ${\bar f}: F(X) \to G$ (respectively, ${\bar f}: A(X) \to G$) such that $f={\bar f} \circ i$.
\end{enumerate}
\end{definition}
The topological  groups $F(X)$ and $A(X)$ always exist and are essentially unique. Note that the mapping $i$ is a topological embedding \cite{Mar, Gra}. If $X$ is a discrete space, it is clear that $F(X)$ and $A(X)$ are also discrete. It is known (see \cite{MMO}) that for each $\mathcal{MK}_\omega$-space $X$, the groups  $F(X)$ and $A(X)$ are also $\mathcal{MK}_\omega$-spaces  and hence sequential.

Analogously we can define free lcs (see \cite{Mar, Rai}):
\begin{definition}
Let $X$ be a Tychonoff space. The {\em  free lcs} $L(X)$ on  $X$ is a pair consisting of a lcs $L(X)$ and  a continuous mapping $i: X\to L(X)$ such that every  continuous mapping $f$ from $X$ to a lcs $E$ gives rise to a unique continuous linear operator ${\bar f}: L(X) \to E$  with $f={\bar f} \circ i$.
\end{definition}
Also the free lcs $L(X)$  always exists and is  unique. The set $X$ forms a Hamel basis for $L(X)$, and  the mapping $i$ is a topological embedding \cite{Rai, Flo1, Flo2, Usp}.
The identity map $id_X :X\to X$ extends to a canonical homomorphism $id_{A(X)}: A(X)\to L(X)$. It is known that $id_{A(X)}$ is an embedding of topological groups \cite{Tkac, Usp2}.
For example, if $X$ is a finite space of cardinality $n$, then $L(X)\cong \mathbb{R}^n$; and if $X$ is a countably infinite discrete space, then $L(X)\cong\phi$, where $\phi$ is the countable inductive limit of the increasing sequence $(\mathbb{R}^k)_{k\in\NN}$.


It is well-known that the space $L(X)$ admits a canonical continuous monomorphism $L(X)\to C_c (C_c (X))$. If $X$ is a $k$-space, this monomorphism is an embedding of lcs \cite{Flo1, Flo2, Usp}. So, for  $k$-spaces, we obtain the next chain of topological embeddings:
\begin{equation} \label{emb}
A(X) \hookrightarrow L(X) \hookrightarrow C_c (C_c (X)).
\end{equation}

Denote by $Q=[0,1]^\NN$ the Hilbert cube. Since $Q$ is a subspace of $\mathbb{R}^\NN$, $Q$ has a $\GG$-base at each its point satisfying the condition $\dd$.
Below we prove Theorem \ref{tFree}.

\begin{proof}[Proof of Theorem \ref{tFree}]
(i) Since $X$ is an  $\mathcal{MK}_\omega$-space,  the space $C_c(X)$ is a Polish space  by
\cite[4.2.2 and 5.8.1]{mcoy}. Thus $C_c (C_c(X))$ has a $\mathfrak{G}$-base satisfying the condition $\dd$ by  Theorems 2 and 9 of  \cite{GKL2} and is a $\Pp_0$-space by Theorem \ref{c-Banakh}. Now (\ref{emb}) implies that  $L(X)$ and $\phi =L(\NN)$ also have a $\mathfrak{G}$-base satisfying the condition $\dd$ and  are  $\Pp_0$-spaces. As  the groups $A(X)$ and $F(X)$ are $\mathcal{MK}_\omega$-spaces,  they embed into $\phi\times Q$ by \cite{Sakai-K}. Thus $A(X)$ and $F(X)$ also have a $\mathfrak{G}$-base satisfying the condition $\dd$ and  are $\Pp_0$-spaces (see \cite{Banakh,GKL}).

(ii) It is well-known that, if $X$ is not discrete, then $A(X)$ and $F(X)$ are not even Fr\'echet-Urysohn. Now Proposition 2.4 and Corollary 3.14 of \cite{GKL} and (i) imply
\[
\begin{split}
\mathfrak{d} & \leq \min\{\chi(A(X)),\chi(F(X)), \chi(L(X))\} \\
& \leq \max\{\chi(A(X)),\chi(F(X)), \chi(L(X))\} \leq \mathfrak{d},
\end{split}
\]
that proves (a). Now let $X$ be discrete, clearly $\chi(A(X))=\chi(F(X))=1$. If $X$ is finite, then $L(X)=\mathbb{R}^{|X|}$ is metrizable, and hence $\chi(L(X))=\aleph_0$. If $X$ is infinite, then $X$ is countably infinite as an $\mathcal{MK}_\omega$-space. So $L(X)=\phi$. Now Proposition 2.4 and Corollary 3.14 of \cite{GKL} imply that $\chi(L(X))=\mathfrak{d}$.
\end{proof}
Note (see \cite{Gab-MSJ}) that for a metrizable space $X$, the space $L(X)$ is a \mbox{$\Pp_0$-space} if and only if $L(X)$ has countable tightness if and only if $X$ is separable.

At the end of this section we consider some applications  to topological vector spaces.


Recall  that a topological space $X$ has the {\it property $\left( \alpha_{4}\right) $ at a point $x\in X$} if for any $\{x_{m,n}:\left( m,n\right) \in \mathbb{N}\times \mathbb{N}\}\subset X$ with $\lim_{n}x_{m,n}=x\in X$, $m\in \mathbb{N}$, there exists a sequence $\left( m_{k}\right) _{k}$ of distinct natural numbers and a sequence $\left( n_{k}\right) _{k}$ of natural numbers such that $\lim_{k}x_{m_{k},n_{k}}=x$; $X$ has the {\it property $\left( \alpha_{4}\right) $} or is an {\it $\left( \alpha_{4}\right) $-space} if it has the property $\left( \alpha_{4}\right)$ at each point $x\in X$. Nyikos proved in \cite[Theorem 4]{nyikos} that any Fr\'{e}chet-Urysohn topological group satisfies $\left( \alpha_{4}\right)$. However there are Fr\'{e}chet-Urysohn topological spaces which do not have $\left( \alpha_{4}\right)$.
Further, in \cite[Lemma 1.3]{ChMPT} it was shown that for a Fr\'{e}chet-Urysohn topological group $G$ the property $\left( \alpha_{4}\right)$ can  be strengthened by the  double sequence property (AS):
\begin{itemize}
\item[{\rm (AS)}] For any family $\{x_{n,k} : (n,k)\in \NN\times\NN \}\subseteq G$, with $\lim_n x_{n,k} =x\in G, k=1,2,\dots,$ it is possible to choose strictly increasing sequences of natural numbers $(n_i)_{i\in\NN}$ and $(k_i)_{i\in\NN}$, such that $\lim_i x_{n_i, k_i} =x$.
\end{itemize}


For a group $G$, $g\in G$ and $n\in\NN$, we set $g^n := g\cdots g$ ($n$ times) and, if $G$ is abelian, $ng:= g+\cdots +g$. Let $G$ be a topological group with $(AS)$ and a sequence $(g_n)$ in $G$ converge to the unit $e$. Clearly, $\lim_{n} g^m_n =e$ for every $m\in\NN$. Applying $(AS)$ to the family $\{ (g_n^m): m\in\NN\}$ of the powers of the sequence $(g_n)$  we propose the following property $(PS)$ which is weaker than $(AS)$.
\begin{definition}
We say that a topological group $G$ has {\em the property $(PS)$} if for every sequence $(g_n)_{n\in\NN} \subseteq G$ converging to the unit $e$ there are strictly increasing sequences $(m_k)$ and $(n_k)$ of natural numbers such that $g^{m_k}_{n_k} \to e.$
\end{definition}

Lemma 1.3 of \cite{ChMPT} immediately implies
\begin{proposition} \label{p-C4}
Any Fr\'{e}chet-Urysohn topological group $G$ has $(PS)$.
\end{proposition}


As usual we denote by $\sigma(E,E')$ the weak topology of a locally convex space $E$. Recall that an abelian topological group $G$ is {\em maximally almost periodic} (MAP) if its continuous characters separate the points of $G$. A MAP abelian group $G$ endowed with the Bohr topology we denote by $G^+$. Recall also that a MAP abelian group $G$ (respectively, a lcs $E$) has the {\em Schur property} if $G^+$ and $G$ (respectively, $(E, \sigma(E,E'))$ and $E$) have the same set of convergent sequences. The next corollary shows that the class of topological groups having $(PS)$ is much wider than the class of Fr\'{e}chet-Urysohn topological group. Proposition \ref{p-C4} applies to get the following
\begin{corollary} \label{p-C4-Schur}
Let $(G,\tau)$ be a Fr\'{e}chet-Urysohn topological group (respectively, lcs) with the Schur property. Then $G^+$ (respectively, $(G,\sigma(G,G'))$) has $(PS)$.
\end{corollary}

Recall (see \cite{Gao}) that a family $\Nn$  of subsets of a topological space $X$ is called a {\em $cs^\ast$-network at  a point} $x\in X$ if for each sequence $(x_n)_{n\in\NN}$ in $X$ converging to  $x$ and for each neighborhood $O_x$ of $x$ there is a set $N\in\mathcal{N}$ such that $x\in N\subseteq O_x$ and the set $\{n\in\NN :x_n\in N\}$ is infinite; the smallest size $|\Nn|$ of a $cs^\ast$-network at $x$ is called the {\em $cs^\ast$-character of $X$ at the point $x$}. The cardinal $cs^\ast_\chi(X)=\sup\{ cs^\ast_\chi(X,x): x\in X\}$ is called the {\em $cs^\ast$-character} of  $X$. It is easy to see that if a topological space $X$ has the strong Pytkeev property, then $X$ has countable $cs^\ast$-character.

Now we apply $(PS)$ for the important class of topological vector spaces  having a fundamental sequence of bounded sets.
\begin{proposition}\label{t-C4}
Let $E$ be a tvs with a fundamental sequence of bounded sets. If $E$ has $(PS)$, then $E$ has countable $cs^\ast$-character.
\end{proposition}

\begin{proof}
Let $\{ D_n\}_{n\in\NN}$ be  a fundamental sequence of closed absolutely convex bounded sets in $E$. We claim that the family $\mathcal{N}:=\{ \frac{1}{k} D_n: k,n\in\NN\}$ is a countable
$cs^\ast$-network at zero. Indeed, let $x_n\to 0$. Choose strictly increasing sequences $(m_k)$ and $(n_k)$ of natural numbers such that $m_k x_{n_k} \to 0.$
Fix an open neighborhood $U$ of zero. Take $a\in\NN$ such that $(m_k x_{n_k}) \subseteq D_a$, and choose $b\in\NN$ such that $D_a \subseteq bU$. Then, by  $\frac{1}{b}D_a \subseteq U$, we have
\[
x_{n_k} =\frac{b}{m_k} \cdot \left( \frac{m_k}{b} x_{n_k} \right) \subseteq \frac{b}{m_k} \left(\frac{1}{b} D_a\right) \subseteq \frac{1}{b} D_a
\]
for every $k \geq k_0$, where $b/m_{k_0} <1$. Thus $(x_n)\cap \frac{1}{b} D_a$ is infinite. 
\end{proof}
For a lcs $E$  with a fundamental sequence of bounded sets which  is also {\em qusibarrelled} (this implies that $E$ is a $(DF)$-space, see \cite{PB}) it is known essentially more than in Proposition  \ref{t-C4}.

\begin{remark} \label{re} {\em
Any $(DF)$-space $E$,  by definition,  admits a fundamental sequence of bounded sets and must be $\aleph_{0}$-quasibarrelled, see \cite{kak}.
On the other hand, by Cascales-K{\c{a}}kol-Saxon result, see  Lemma 15.2 from \cite{kak}, it follows that any quasibarrelled $(DF)$-space $L$ admits a $\mathfrak{G}$-base. So $L$ has  the strong Pytkeev property by Theorem 9 of \cite{GKL2}, and hence $L$ has countable $cs^\ast$-character.
We recall also  that the strong dual $F'$ of any Fr\'{e}chet space $F$ is a $(DF)$-space, see \cite{kak}, and $F'$ is quasibarrelled if and only if $F$ is a distinguished space, see  \cite{bierstedt2}. }
\end{remark}

Recall that a tvs $E$ is {\em locally bounded} if $E$ has a bounded neighborhood of zero $U$. Then $E$ is metrizable and $(n^{-1}U)_{n}$ forms a countable base of neighborhoods of zero for $E$. Recall also that for every locally bounded topological vector space $E$ there exists a $p$-norm for some $0<p\leq 1$   generating the original vector topology of $E$, see \cite[Theorem 6.8.3]{jarchow}.

\begin{proof}[Proof of Theorem \ref{t-Metr-TVS-FU}]
We need to show only sufficiency. Let $E$ be a Fr\'echet-Urysohn tvs and let $(B_{n})_{n}$ be a fundamental (increasing) sequence of balanced bounded sets of $E$. Applying Propositions \ref{p-C4} and \ref{t-C4} and  \cite[Theorem 3]{BZ} (stating that every topological groups which is Fr\'echet-Urysohn and has countable $cs^\ast$-character is metrizable) we conclude that $E$ is metrizable. We claim that $E$ is locally bounded. Indeed, let $(U_{n})$ be a decreasing base of balanced neighborhoods of zero in $E$. We show that some $B_{n}$ is a (bounded) neighborhood of zero. If this is not the case  we have $U_{n}\nsubseteq nB_{n}$ for each $n\in\mathbb{N}$. For each $n\in\mathbb{N}$ choose $ n^{-1}x_{n}\in U_{n}\setminus B_{n}$. But then the  set $B:=\{n^{-1}x_{n}:n\in\mathbb{N}\}$ is bounded and not included in any $B_{n}$, a contradiction. The proof is completed if we apply the notice above concerning $p$-normed spaces.
\end{proof}

Note that a {\em metrizable} lcs $E$ is normable if and only if the strong dual $E'$ of $E$ is a Fr\'{e}chet-Urysohn lcs by \cite[Theorem 2.8]{ChMPT} (this result can be derived also from \cite{CKS}).
\begin{corollary} \label{fre}
Let $E$ be the topological product of a family $(E_{j})_{j\in J}$ of metrizable topological vector spaces. Then $E$ has a fundamental sequence of bounded sets if and only if $J$ is finite.
\end{corollary}
\begin{proof}
Let $E_{0}$ be the $\Sigma$-product of $E$, i.e. \[
E_{0}:=\{(x_{i})\in E: |j\in J:x_{j}\neq 0|\leq\aleph_{0}\}.
\]
It is known (see \cite{Nob}) that $E_{0}$ is Fr\'echet-Urysohn. Assume that $E$ has a fundamental sequence $(B_{n})_{n}$ of bounded sets, so $E_{0}$ admits such a sequence, too.  Theorem \ref{t-Metr-TVS-FU} applies to deduce that $E_{0}$ is metrizable, hence $E$ is metrtizable (since $E_{0}$ is dense in $E$). Then some $B_{m}$ is a bounded neighborhood of zero in $E$, see the proof of Theorem \ref{t-Metr-TVS-FU}. This implies that $J$ is finite, otherwise $E$ would contain $\mathbb{K}^{\mathbb{N}}$ ($\mathbb{K}$ the field of either real or complex numbers) having a bounded neighborhood of  zero, which is impossible.
\end{proof}

\begin{remark}\label{rema} {\em
If in Corollary \ref{fre} each $E_{j}$ is additionally complete and $J$ is uncountable, $E_{0}$ is a Baire nonmetrizable subspace of $E$. By Theorem \ref{tMetr-cn} the space $E_{0}$ does not have a countable $cn$-character. }
\end{remark}

The following lemma supplements Remark \ref{re}.
\begin{lemma} \label{p-Appl}
Let $(E,\tau)$ be a lcs  with a fundamental sequence  $(B_{n})_{n}$ of bounded sets. Then  $E$ endowed with the finest locally convex topology $\xi$ having the same bounded sets as $\tau$  (which exists) has a $\mathfrak{G}$-base and has the strong Pytkeev property. In particular, if $(E,\tau)$ is bornological, then $\tau=\xi$.
\end{lemma}
\begin{proof}
We may assume that all sets $B_{n}$ are absolutely convex. For each $n\in\mathbb{N}$ let $E_{n}$ be the linear span of $B_{n}$ endowed with the Minkowski functional norm topology. Let $(E,\xi)$ be the strict inductive limit space of the sequence $(E_{n})_{n}$ of normed spaces.   Then $(E,\xi)$ is bornological, i.e. every absolutely convex bornivorous set  in $(E,\xi)$ is a $\xi$-neighborhood of zero, and the topologies $\xi$ and $\tau$ have the same bounded sets. Then clearly  the topology  $\xi$ is the finest one as we claimed. By the proof of Theorem 3 of \cite{GKL2}  the space $(E,\xi)$ has a $\mathfrak{G}$-base and the strong Pytkeev property.  Finally, if $(E,\tau)$ is bornological, then (by definition) we have $\tau=\xi$.
\end{proof}

\begin{proposition}
Let $(E,\tau)$ be a lcs such that:
\begin{enumerate}
\item[{\rm (i)}] $E$ admits a fundamental sequence  $(B_{n})_{n}$ of absolutely convex  bounded sets.
\item[{\rm (ii)}] Every linear functional $f$ over $E$ is continuous if and only if any restriction $f|_{B_{n}}$ is continuous.
\end{enumerate}
Then $E$ with the weak topology $\sigma(E,E')$  is angelic and a $\sigma(E,E')$-compact set $K$ is $\sigma(E,E')$-metrizable if and only if $K$ is contained in a $\sigma(E,E')$-separable subset of $E$.
\end{proposition}

\begin{proof}
Let $\xi$ be the locally convex topology as in Lemma \ref{p-Appl}.  Observe that  $\tau$ and $\xi$ have the same continuous linear functionals; hence the both spaces $(E,\tau)$ and $(E,\xi)$ have the same weak topology. Since $(E,\xi)$ has a $\mathfrak{G}$-base by Lemma \ref{p-Appl}, we apply  Cascales-Orihuela's  results, see   \cite[Proposition 11.3]{kak} and \cite[Corollary 9]{fkm} to complete the proof.
\end{proof}



\begin{example} \label{exa-Cc(b)} {\em
Let $X=[0,\mathfrak{b})$, where $\mathfrak{b}$ is the small uncountable cardinal equal to the smallest cardinality of a subset of $\NN^\NN$ which cannot be covered by a $\sigma$-compact subset of $\NN^\NN$. It is well-known that the cofinality $\mathrm{cf}(\mathfrak{b})$ of $\mathfrak{b}$ is uncountable, and hence $X$ is not Lindel\"{o}f. The space $C_c(X)$ has a $\GG$-base by Proposition 16.14 of \cite{kak}; in particular, $C_c(X)$ has countable $cs^\ast$-character by Theorem 3.12 of \cite{GKL}. On the other hand, since $X$ is not Lindel\"{o}f,  $C_c(X)$ has uncountable tightness by a result of McCoy (see \cite[Lemma 16.4]{kak}). Taras Banakh noted that the pseudocharacter $\psi(C_c(X))$ of $C_c(X)$ is uncountable (this follows from the inequality $\mathrm{cf}(\mathfrak{b})> \aleph_0$ and the Ferrando-K{\c{a}}kol duality theorem, see \cite{feka}), and  hence $C_c(X)$ is not submetrizable that gives a negative answer to Question 2.15 of \cite{GKL}. Finally, we note that the space $C_c(X)$ is not  a $\sigma$-space because $\psi(C_c(X))>\aleph_0$ (see \cite[4.3]{gruenhage}). }
\end{example}

\bibliographystyle{amsplain}

\end{document}